\documentclass[11pt]{amsart}

\usepackage{amsmath,amsfonts,amssymb,amsthm,graphicx,float,color,tikz}

\usepackage[margin=1.2in]{geometry}

\usepackage[T1]{fontenc}
\usepackage{lmodern}

\usepackage{longtable}

\newcommand{\bb}{\mathbb}
\newcommand{\C}{\bb C}
\newcommand{\h}{\bb H}
\newcommand{\R}{\bb R}

\newcommand{\G}{\bb G}
\newcommand{\s}{\bb S}

\newcommand{\vv}{\mathbf{v}}
\newcommand{\hh}{\mathcal H}
\newcommand{\kk}{\underline{k}}

\newcommand{\area}{\operatorname{area}}
\newcommand{\vol}{\operatorname{vol}}
\newcommand{\slope}{\operatorname{slope}}
\newcommand{\Real}{\operatorname{Re}}
\newcommand{\Imag}{\operatorname{Im}}
\newcommand{\ath}{\operatorname{ath}}
\newcommand{\Li}{\operatorname{Li}}

\newcommand{\SL}{\operatorname{SL}}

\newcommand{\vect}[2]{%
  \Bigl(\begin{array}{@{}c@{}}#1\\ #2\end{array}\Bigr)%
}

\newcommand{\mat}[4]{%
  \begin{bmatrix}#1 & #2\\ #3 & #4\end{bmatrix}%
}

\newcommand{\om}{\omega}
\newcommand{\omL}{\omega_\mathrm{L}}
\newcommand{\OmX}{\Omega_X}
\newcommand{\Om}{\Omega}
\newcommand{\La}{\Lambda}

\newcommand{\alphabar}{\overline{\alpha}}
\renewcommand{\phi}{\varphi}
\newcommand{\phibar}{{\overline{\phi}}}

\newcommand{\minuszero}{\backslash\{0\}}

\newcommand{\x}{\mspace{1mu}}

\newcommand{\tb}{{\scriptsize\textbullet}}
\newcommand{\tbx}{{\scriptsize\textbullet}\phantom{--}}

\renewcommand{\quad}{\hspace*{2mm}}

\newcommand{\psp}{\par\smallskip\par}
\newcommand{\pmp}{\par\medskip\par}

\newcommand{\appendixmode}{%
  \setcounter{section}{0}
  \renewcommand{\thesection}{\Alph{section}}%
}

\newtheorem{Theorem}{Theorem}
\numberwithin{Theorem}{section}
\newtheorem{lemma}[Theorem]{Lemma}
\newtheorem*{ques}{Question}

\theoremstyle{remark}
\newtheorem{Remark}{Remark}
\newtheorem*{Notation}{Notation}

\numberwithin{equation}{section}

\begin{document}

\title[Golden L slope gap distribution]
{The gap distribution of slopes on the golden L}
\author{Jayadev S.~Athreya}
\author{Jon Chaika}
\author{Samuel Leli\`evre}
\email{jathreya@illinois.edu}
\email{chaika@math.utah.edu}
\email{samuel.lelievre@gmail.com}
\address{%
  Department of Mathematics,
  University of Illinois Urbana-Champaign,
  1409 W.\ Green Street,
  Urbana, IL 61801, USA%
}
\address{
  Department of Mathematics,
  University of Utah,
155 South 1400 East, JWB 233
Salt Lake City, Utah 84112-0090 %
}
\address{%
  Laboratoire de math\'ematique d'Orsay,
  UMR 8628 CNRS / Universit\'e Paris-Sud, 
  B\^atiment 425,
  91405 Orsay cedex,
  France%
}

\thanks{J.S.A.\ partially supported by NSF grant DMS 1069153, and NSF grants DMS 1107452, 1107263, 1107367 ``RNMS: GEometric structures And Representation varieties" (the GEAR Network).}
\thanks{J.C.\ partially supported by NSF postdoctoral fellowship DMS 1004372}
\thanks{S.L.\ partially supported by ANR projet blanc GeoDyM}

\begin{abstract}
We give an explicit formula for the limiting gap distribution
of slopes of saddle connections on the golden L,
or any translation surface in its $\SL(2, \R)$-orbit,
in particular the double pentagon.
This is the first explicit computation of the distribution of gaps
for a flat surface that is not a torus cover.
\end{abstract}

\maketitle

%
%
%


\section{Introduction}
\label{sec:intro}

\subsection{The golden L}
\label{subsec:golden}

The golden L is a \emph{translation surface}
obtained from an L-shaped polygon (with length ratios
equal to the golden ratio $\phi = \frac{1+\sqrt{5}}{2}$)
by gluing opposite sides by horizontal and vertical translations
(see Figure~\ref{fig:golden:L}).
It has genus two and a single cone-type singularity
of angle $6\pi$ resulting from the identification of all vertices
of the L-shaped polygon to a single point after the side gluings.

In this paper, we describe an explicit computation of the distribution
of gaps between slopes of \emph{saddle connections} on the golden L,
where a saddle connection is a straight line trajectory
starting and ending at the cone point of the golden L.
This can be viewed as a geometric generalization of
the gap distribution for Farey fractions~\cite{Athreya}.

\begin{figure}[ht]

\begin{tikzpicture}[scale=2]

%
%
%

%
%
%

\tikzstyle{every circle node}=[fill=black,inner sep=1.25pt]

\begin{scope}[thick] 
\draw (0,0) -- (1,0) node [circle] {};
\draw (1,0) -- (1.618, 0) node [circle] {};
\draw (1.618, 0) -- (1.618, 1) node [circle] {};
\draw (1.618, 1) -- (1,1) node [circle] {};
\draw (1,1) -- (1, 1.618) node [circle] {};
\draw (1, 1.618) -- (0, 1.618) node [circle] {};
\draw (0, 1.618) -- (0, 1) node [circle] {};
\draw (0, 1) -- (0,0) node [circle] {};
\end{scope}


\draw [very thin] (0,0) -- (1.618, 0.5);
\draw [very thin] (0,0.5) -- (1.618, 1);

\begin{scope}[xshift=-1.25mm,<->,shorten <=1pt,shorten >=1pt]
\draw (0,0) -- node [left] {\footnotesize \strut$1$} (0,1);
\draw (0,1) -- node [left] {\footnotesize \strut$\phibar$} (0,1.618);
\end{scope}

\begin{scope}[yshift=-1.25mm,<->,shorten <=1pt,shorten >=1pt]
\draw (0,0) -- node [below=-1mm] {\footnotesize \strut$1$} (1,0);
\draw (1,0) -- node [below=-1mm] {\footnotesize \strut$\phibar$} (1.618,0);
\end{scope}

\end{tikzpicture}

\caption{Left: The golden L has sides glued pairwise by 
horizontal and vertical translations. All vertices, marked by dots, become a
single point after side gluings.
A saddle connection crossing a pair of identified sides is shown.}
\label{fig:golden:L}

\end{figure}
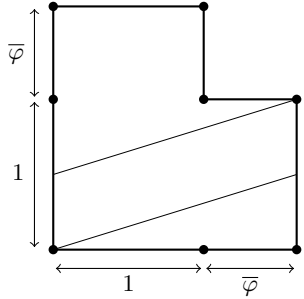

\subsection{Translation structure}
\label{subsec:trans}

The side identifications of the golden L are by translations, which are
holomorphic, so it inherits a Riemann surface structure, as well as a
holomorphic one-form (from the form $dz$ in the plane,
which is preserved by translations).
This one-form $\om$ on this Riemann surface has a single zero,
of order two, at the cone-point of the golden L.

We identify the golden L to the one-form $\omL$,
the underlying Riemann surface being implied.
The golden L and the closely related double pentagon have been
popular objects of study and a testing ground for various properties
of translation surfaces~\cite{
 Davis, DFT}.

\subsection{Saddle connections and holonomy}
\label{subsec:sc} 

Associated to each oriented saddle connection $\gamma$ is
a \emph{holonomy vector}, which we will call a saddle connection vector,
$$
\vv_{\gamma} = \int_{\gamma} \om \in \C,
$$
which records how far and in what direction $\gamma$ travels.
The set of saddle connection vectors,
$$
\La_\om = \{\,\vv_{\gamma}: \gamma \mbox{ an oriented saddle connection
on the golden L}\,\} \subset \C,
$$
is a discrete subset of the plane with
\emph{quadratic asymptotics}~\cite{Masur}.  Veech~\cite{Veech} showed that
$$
|\La_\om \cap B(0, R)| \simeq c \x R^2,
$$
where $a(R) \simeq b(R)$ indicates that the ratio
of $a(R)$ and $b(R)$ goes to $1$ as $R \to \infty$,
and where $c = 3\x\pi/10$ is the volume of $\h^2/\Gamma$,
where $\Gamma = \triangle(2, 5, \infty)$ is the Hecke $(2, 5, \infty)$
triangle group, which is the \emph{Veech group}
of the golden L (see \S\,\ref{subsec:veech}).


\subsection{Slopes and uniform distribution}
\label{subsec:gap}

The object of this paper is to study the distribution of
the set of \emph{slopes} of $\La_\om$.
Since the set $\La_\om$ is symmetric about the  coordinate axes
as well as about the first and second diagonals,
it is enough to study slopes of vectors in the first quadrant below
the first diagonal,
$$
\s = \left\{ \slope(z) = \frac{\Imag(z)}{\Real(z)}:
z \in \La_\om,\ 0 \le \Imag(z) \le \Real(z)\right \} \subset [0, 1].
$$
We view $\s$ as the union of the nested sets
$$
\s_R = \left\{ \slope(z) = \frac{\Imag(z)}{\Real(z)}:
z \in \La_\om,\ 0 \le \Imag(z) \le \Real(z) \le R \right\}.
$$
In~\cite{Veech}, Veech shows that not only the cardinality
$N(R) = |\s_R|$ grows quadratically (as discussed above),
but also the sets $\s_R$ become \emph{equidistributed}
in $[0,1]$ with respect to Lebesgue measure.
That is, the uniform probability measure on the finite set $\s_R$
weak-$*$-converges to the Lebesgue probability measure on $[0, 1]$:
$$
\frac{1}{N(R)}\sum_{s \in \s_R} \delta_s \xrightarrow{\text{w}*} \lambda
$$
(where $\delta_s$ denotes the Dirac mass at $s$).
This result can be interpreted as saying that to the first order,
the directions of saddle connections on the golden L appear randomly.

\subsection{Gap distributions}

A finer assessment of randomness arises from
the \emph{gap distribution} of the slopes.
For $R \geq 1$ (so $\s_R$ is nonempty), index the elements of $\s_R$ 
in increasing order:
$$
0=s_R^{(0)} < s_R^{(1)} < s_R^{(2)} < \ldots < s_R^{(N(R)-1)} = 1
$$
and consider the set of scaled differences or \emph{gaps} (scaled by $R^2$
since $N(R)$ grows in $R^2$):
$$
\G_R =\left \{ R^2\left(s_R^{(i+1)}-s_R^{(i)}\right):
0 \le i < N(R)-1\right\}.
$$
We are interested in the limiting behavior of the probability measure
supported on $\G_R$, in particular, for $0 \leq a < b \leq +\infty$,
the existence and evaluation of
\begin{equation}
\label{eq:ell:gap}
\lim_{R \rightarrow \infty} \frac{\left| \G_R
\cap (a, b)\right|}{N(R)}.
\end{equation}
If the slopes were `truly random', obtained from sampling a
sequence of independent identically distributed random variables
following the uniform law on $[0,1]$, the associated gap distribution
would be \emph{exponential}, that is, the above limit would be
$e^{-a} - e^{-b}$.

\medskip

Our main result is the existence and computation of the gap distribution
(\ref{eq:ell:gap}). 

\begin{Theorem}
\label{theorem:main}
There is a limiting probability distribution function 
$f: [0, \infty) \rightarrow [0,\infty)$, with
$$
\lim_{R \rightarrow \infty}
\frac{\left| \G_R \cap (a, b)\right|}{N(R)}
= \int_{a}^{b} f(x) dx.
$$
This function $f$ is continuous, piecewise real-analytic, with seven points of
non-differentiability, and the real-analytic pieces have explicit expressions
involving usual functions.
\end{Theorem}

\begin{figure}[ht]
\includegraphics[width=0.49\textwidth, height = 45mm]{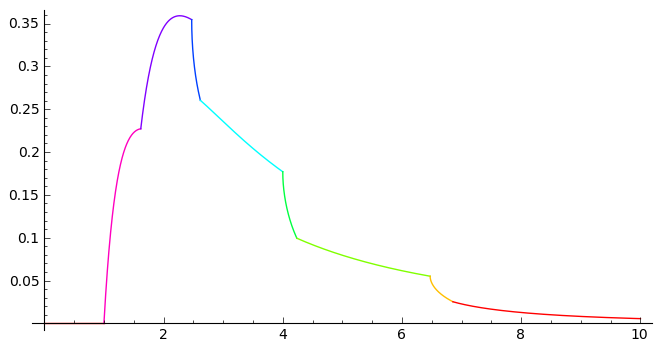}
\includegraphics[width=0.49\textwidth, height = 45mm]{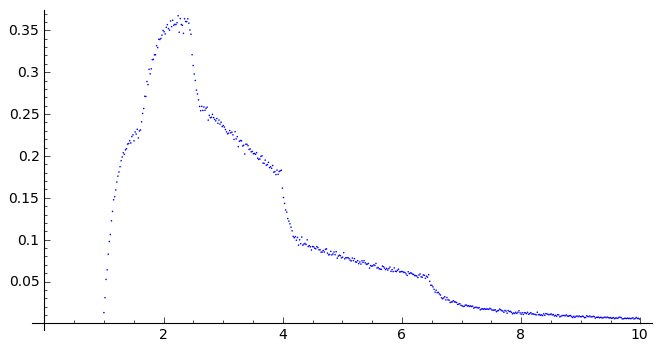}
\caption{The limiting and empirical distributions for gaps 
of saddle connection slopes on the golden L. The empirical distribution is 
for saddle connection vectors of slope at most $1$
and horizontal component less than $10^4$.}
\label{fig:goldenLdist}
\end{figure}

\medskip
\noindent
\begin{Remark}
The formulas for the real-analytic pieces of
the probability distribution function $f$ are given in
Appendix~\ref{appendix:formula}.
\end{Remark}
\begin{Remark}
The distribution has no support at $0$, in
fact $f(x) = 0$ for $0 \le x \le 1$.  The tail of the distribution is
\emph{quadratic}, that is, for $t \gg 0$, $$\int_{t}^{\infty} f(x) dx \sim
t^{-2},$$ where $\sim$ indicates that, for large enough $t$,
the ratio is bounded between two positive constants.  In particular,
as is clear from Figure~\ref{fig:goldenLdist},
the distribution is \emph{not} exponential.
\end{Remark}

\subsection{Typical surfaces}
\label{subsec:typical}

In~\cite{AChaika}, the first two authors considered gap distributions for
\emph{typical} translation surfaces, that is, surfaces in a set of full
measure for the Masur-Veech measure (or, in fact, any ergodic
$\SL(2, \R)$-invariant measure) on a connected component of a stratum of
the moduli space $\Om_g$ of genus $g \geq 2$ translation surfaces.

It was shown that the limiting distribution
exists, and is the same for almost every surface, and, as above, the tail is
quadratic.  In contrast to our setting, the distribution does have support at
$0$ (for generic surfaces for the Masur-Veech measure).  For \emph{lattice
surfaces}, of which the golden L is an example, the distribution does
\emph{not} have support at $0$, in fact, there are \emph{no} small gaps.  This
is essentially equivalent to the \emph{no small triangles} condition of
Smillie-Weiss~\cite{nosmalltri}.  However, the only explicit computations
in~\cite{AChaika} were for branched covers of tori, and relied on previous
work of Marklof-Str\"ombergsson~\cite{MS} on the space of affine lattices.  To
the best of our knowledge, the current paper gives the first explicit
computation of a gap distribution of saddle connections for a surface which is
not a branched cover of a torus.

\subsection{Strategy of proof}
\label{subsec:strategy}

We follow the work of~\cite{ACheung} where the first author and
Y.~Cheung, inspired by the work of Boca-Cobeli-Zaharescu~\cite{BCZ}
(also see~\cite{BZsurvey} for a survey),
gave an ergodic-theoretic proof of Hall's Theorem~\cite{Hall}
on the gap distribution of Farey fractions, using the horocycle flow
on the modular surface.
The key idea is to construct a Poincar\'e section for the horocycle flow
on an appropriate moduli space, and to compute the distribution of
the return time function with respect to an appropriate measure.
This strategy can be used in many different situations,
see~\cite{Athreya} for a description of some of them.


\section{Strata, $\SL(2, \R)$-action, and Veech groups}
\label{sec:strata}

In this section we relax the notations $\om$, $\La_\om$ from their
use in the previous section.

\subsection{Strata of translation surfaces}
\label{subsec:strata}

A (compact, genus $g$) \emph{translation surface} is given by a holomorphic
one-form $\om$ on a compact genus $g$ Riemann surface.  Leaving the
Riemann surface implicit, we simply refer to $\om$ as the translation surface.
More geometrically, a translation surface is given by a union of polygons
$P_1 \cup \dots \cup P_n \subset \C$, and gluings of parallel sides
by translations, such that each side is glued to exactly one other,
and the total angle at each vertex class is an integer multiple of $2 \pi$.
Since translations are holomorphic, and preserve $dz$, we obtain
a complex structure and a holomorphic differential on the glued up surface.
The zeros of the
differential are at the identified vertices with total angle greater than
$2\pi$.  A zero of order $k$ is a point with total angle $2\pi(k+1)$.  The sum
of the angle excess is $2\pi(2g-2)$,
where $g$ is the genus of the glued up surface. Equivalently the orders of 
the zeros sum up to $2g-2$.

Thus, the space of genus $g$ translation surfaces can be stratified by integer
partitions of $2g-2$.  If $\kk =(k_1, \ldots, k_s)$ is a
partition of $2g-2$, we denote by $\hh(\kk)$ the moduli space of
translation surfaces $\om$ such that the multiplicities of the zeros are
given by $k_1, \dots, k_s$.  The golden L is a genus $2$ surface
with one zero of order $2$ (total angle $6\pi$), that is, $\om \in
\hh(2)$.

\subsection{Action of $\SL(2, \R)$}
\label{subsec:sl2}

The group $\SL(2,\R)$ acts on each stratum $\hh(\kk)$ via its action
by linear maps on the plane: given a surface $\om$ glued from polygons
$P_1$, $\dots$, $P_n$ in the plane, define $g\cdot \om$ as
$g P_1 \cup \dots \cup g P_n$,
with the same side gluings for $g \cdot \om$ as for $\om$.

\subsection{Veech groups and lattice surfaces}
\label{subsec:veech}

For a partition $\kk$ of $2g-2$, with $g\geq 2$, and a typical surface
$\om \in \hh(\kk)$, the
stabilizer $\Gamma$ of $\om$ in $\SL(2, \R)$, known as the \emph{Veech group}
of $\om$, is trivial.  However, for a \emph{dense} subset, the Veech group is
a \emph{lattice} in $\SL(2, \R)$.  These surfaces are known as \emph{lattice
surfaces} or \emph{Veech surfaces}.  While these lattices are never
co-compact, Smillie~\cite{Smillie} showed that the $\SL(2, \R)$-orbit of a
translation surface $\om$ is a \emph{closed} subset of $\hh(\kk)$ if and
only if $\om$ is a lattice surface, and in this setting it can be identified
with the quotient $\SL(2, \R)/\Gamma$.  Thus, a general principle is the
following:

\medskip
\noindent
\textbf{Principle.}
\textit{If a problem about the geometry of a translation surface can be framed
in terms of its $\SL(2, \R)$ orbit, and $\om$ is a lattice surface, then the
problem can be reduced to studying the dynamics of the $\SL(2, \R)$ action on the
homogeneous space $\SL(2, \R)/\Gamma$}.

\subsection{Veech group orbits and saddle connections}
\label{subsec:vgo:and:sc}

Recall that a \emph{saddle connection} $\gamma$ on a translation surface 
$\om$ is a geodesic in the flat metric determined by $\om$ on the underlying 
Riemann surface, with both endpoints at zeros of $\om$ and no zero of $\om$ 
in its interior.
The \emph{holonomy vector} of $\gamma$ is given by
$$
\vv_{\gamma} = \int_{\gamma} \om \in \C.
$$
The set of holonomy vectors
$$
\La_\om = \{\vv_{\gamma}: \gamma \mbox{ is a saddle connection on } \om\}
\subset \C
$$
is a discrete subset of the plane, which varies equivariantly under the
$\SL(2, \R)$-action, that is $ \La_{g\om} = g\La_{\om}.$ When $\om$ is a
lattice surface, the set $\La_{\om}$ is a finite union of orbits of the Veech
group $\Gamma$ acting linearly on $\R^2\minuszero$ (see, for
example~\cite{HSsurvey}).  That is, there are vectors $\vv_1, \ldots, \vv_k
\in \R^2\minuszero$ such that
$$
\La_{\om} = \bigcup_{i=1}^k \Gamma \vv_i.
$$

In the setting of the golden L, we have two vectors
$\vv_1 = \frac{\sqrt5 - 1}{2}$ and $\vv_2 = 1$, viewed as complex numbers.
Since we are interested in the \emph{slopes} of vectors, and $\vv_1$ and
$\vv_2$ are collinear, we work with the orbit $\Gamma \vv_2$, where
$\Gamma = \Gamma(\om) = \triangle(2, 5, \infty)$.  We fix the notation $\Gamma
= \triangle(2, 5, \infty)$ in the rest of this paper, and let $\La = \Gamma
\vv_2 \subset \La_\om$.


\section{Horocycle flows and first return maps}
\label{sec:horocycle} 

In the rest of the paper, $\om$ denotes the golden L, $\Gamma$
denotes its Veech group, the $(2, 5, \infty)$ triangle group, and 
$\La$ denotes the subset $\Gamma \vv_2$ of $\La_\om$ described at
the end of section~\ref{sec:strata}.

\subsection{Horocycle flow and slope gaps}
\label{subsec:horo}
The key tool in the proof of Theorem~\ref{theorem:main} is the construction of
a first return map for the horocycle flow on the $\SL(2, \R)$-orbit of $\om$,
which, as discussed above, can be identified with the homogeneous space
$\SL(2, \R)/\Gamma$.  For our purposes, the horocycle flow is given by the
(left) action of the subgroup:
$$
\left\{h_s = \mat{1}{0}{-s}{1}: s \in \R \right\}.
$$
We also define, for $a, b \in \R$ with $a \neq 0$, the matrix
$$
g_{a, b} = \mat{a}{b}{0}{a^{-1}}.
$$

\subsubsection*{Slope gaps.}
Note that if $z \in \C$, and $\slope(z) = \Imag(z)/\Real(z)$ is the slope,
we have 
$$
\slope(h_s z) = \slope(z) - s,
$$
since $\Real(h_s z) = \Real(z)$, $\Imag(h_s z) = \Imag(z) -s \Real(z)$.
Thus, the action of $h_s$ preserves \emph{differences} in slopes.
Thus to understand the slope gaps of $\La$ (or more generally
of $g\La$ for $g \in \SL(2, \R)$), it is useful to consider
the orbit of $\La$ under $h_s$.

\subsection{A Poincar\'e section}
\label{subsec:return}

Following~\cite{Athreya, ACheung}, we build a cross-section
to the flow $h_s$ on the space $X = \SL(2, \R)/\Gamma$.
Let $\OmX \subset X$ be given by
\begin{equation}
\label{eq:cross}
\OmX = \{ g \Gamma: g \La \cap (0, 1] \neq \emptyset \},
\end{equation}
where we view $(0, 1]$ as the horizontal segment $\{0 < x \le 1 \} \subset \C$.
That is, $\OmX$ consists of surfaces in the $\SL(2, \R)$ orbit of $\om$
which contain a short (length $\le 1$) horizontal vector.

\begin{Theorem}
\label{theorem:cross}
The subset $\OmX \subset X$ is a Poincar\'e section to
the horocycle flow $h_s$ on $X$.
Moreover, the map $(a,b) \mapsto g_{a,b} \Gamma$ establishes a bijection
$$
\Om = \{ (a, b) \in \R^2: 0<a\le 1, 1-a\phi<b\le 1\} \longrightarrow \OmX.
$$
In these coordinates, the return map $T: \Om \rightarrow \Om$ is a
measure-preserving bijection, piecewise linear with countably many pieces.
The return time function $R: \Om \rightarrow \R^+$ defined by 
$$
R(a, b) = \min\{s>0: h_s g_{a,b} \Gamma \in \OmX\}
$$
is a piecewise rational function with three pieces, and is uniformly bounded
below by $1$.  We call the map $T$ the golden-L BCZ map.

\end{Theorem}

\subsection{Connection to slope gaps}
\label{sec:connect}

The connection between Theorem~\ref{theorem:cross} and slope gap distributions
can be seen as follows.  For $g \in \SL(2, \R)$, consider the set of positive
slopes
$\s^g_1 = \{ s_1 < s_2 < \ldots < s_N < \ldots\}$ of elements of $g\La \cap 
U_1$, where  $U_1$ is the vertical strip
$$
U_1 = (0,1] \times (0,+\infty) =  \{ z \in \C: \Imag(z) >0, \Real(z) \in (0,1]\}.
$$
These slopes are the positive times $s$ when the orbit
$\{h_s g \Gamma\}_{s>0}$ intersects $\OmX$.  That is,
$$h_{s_i} g \Gamma \in \OmX,\quad i = 1, 2, \ldots$$
Let $(a, b) \in \Om$ be such that $h_{s_1} g \Gamma = g_{a, b} \Gamma$.
Then we have, for $i \in \mathbb{N}$, $R(T^i(a, b)) = s_{i+2} - s_{i+1}$.
That is, the set of gaps $\G^N_g = \{s_{i+1} - s_i: i=1, 2, 3, \ldots N\}$
is given by the roof function $R$ evaluated along the orbit of
the return map $T$ up till time $N-1$.  For $t >0$, the proportion of
gaps of size at most $t$ can be expressed as a \emph{Birkhoff sum} of the
indicator function $\chi_t$ of the set $R^{-1} ( [t, \infty) ) \subset \Om$,
via
$$
\frac{1}{N} |\G^N_g \cap [t, \infty)| = \frac{1}{N} \sum_{i=0}^{N-1}
\chi_t (T^i (a, b)).
$$
Thus, the dynamics and ergodic theory of the return
map (and in particular the distribution of the roof function along orbits) are
crucial to understanding the gap distributions of slopes.

\subsection{Proof of Theorem~\ref{theorem:cross}}

Let $a_t = g_{e^{t/2}, 0}$ and $u_s = g_{1, s}$.  The action of $a_t$ is
known as the \emph{geodesic} flow and the action of $u_s$ is the (opposite)
horocycle flow.  For $x = g\Gamma \in X$, let $\ell(x) = \min\{ | g \vv|: \vv
\in \La\}$ denote the length of the shortest nonzero vector in $g\La$.  For
any compact subset $K \subset X$, there is an $\epsilon >0$ such that for any $x
\in K$, $\ell(x) > \epsilon$.  Given $x = g\Gamma \in X$, the following are
equivalent (see, e.g.,~\cite{HSsurvey}):
\begin{itemize}
\item[\tb] \textbf{$u_s$-periodicity}: there exists $s_0 \neq 0$ so that $u_{s_0} x = x$.
\item[\tb] \textbf{$a_{-t}$-divergence}:
as $t \to +\infty$, $a_{-t} x \to \infty$ in $X$.  In
particular, $\ell(a_{-t} x) \to 0$.
\item[\tb] $x$ has a \textbf{horizontal saddle connection}, that is,
there exists $\vv \in g \La \cap \R$.
\item[\tb] \textbf{upper triangular form}:
we can write $x = g_{a, b} \Gamma$, where $a = \min \{ |\vv|: \vv \in g\La
\cap \R\}$ is the length of the shortest horizontal vector in $g\La$.
\end{itemize}
Similarly, $h_s$-periodicity is equivalent to
$a_t$-divergence and having a vertical saddle connection.

\medskip

Recall that $\OmX$ is the set of surfaces in the $\SL(2, \R)$-orbit of the
golden L with a length $\le 1$ horizontal long saddle connection vector.  By
the above, these can be expressed as $g_{a, b'} \Gamma \in X$, with $0 < a \le
1$, and $b' \in \R$.  Since the parabolic element $u_{-\phi}$ is in $\Gamma$,
we can apply it $k$ times on the right to $g_{a, b'}$ to obtain $g_{a, b' -
k\phi a}$, where $ b= b'-k\phi a \in (1-\phi a, 1]$.  This condition
determines $b$ uniquely (and given any starting $b'$, it determines $k$).
Conversely, any surface of the form $g_{a, b} \om$ with $(a, b) \in \Om$ has
the saddle connection vector $a \in \C$, and thus clearly has a horizontal
vector of length at most $1$.  Thus, we have shown that $\OmX$ and $\Om$
are in bijection.

%
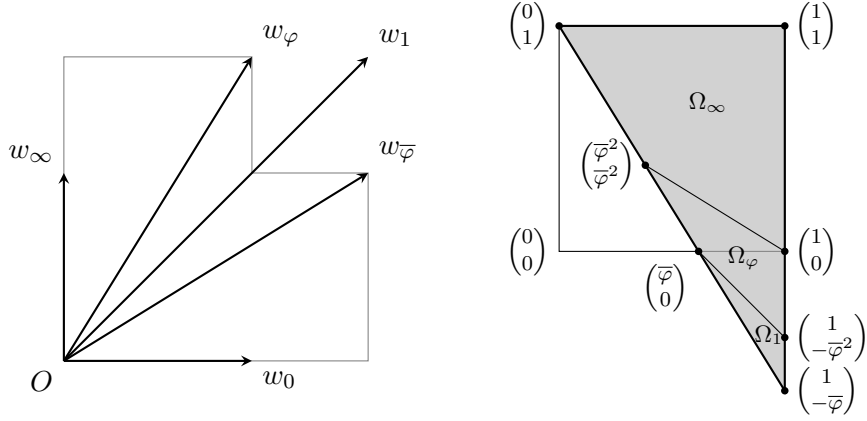
\begin{figure}[ht]

\begin{minipage}{0.4\textwidth}
\centering
\begin{tikzpicture}[scale=2.5]
  \draw  [gray] (0,0)
-- (1,0)
-- (1.618,0)
-- (1.618,1)
-- (1,1)
-- (1,1.618)
-- (0,1.618)
-- (0,1)
-- (0,0)
-- cycle;
  \draw [->,thick,>=stealth] (0,0) node (O) [below left] {$O$} -- (1,1.618) node (wphi) [above right] {$w_\phi$};
  \draw [->,thick,>=stealth] (0,0) -- (1.618,1.618) node (w1) [above right] {$w_1$};
  \draw [->,thick,>=stealth] (0,0) -- (1.618,1) node (wphibar) [above right] {$w_\phibar$};
  \draw [->,thick,>=stealth] (0,0) -- (1,0) node (w0) [below right] {$w_0$};
  \draw [->,thick,>=stealth] (0,0) -- (0,1) node (woo) [above left] {$w_\infty$};
\end{tikzpicture}
\end{minipage}
\begin{minipage}{0.4\textwidth}
\centering
\begin{tikzpicture}[scale=3] 
\tikzstyle{every circle node}=[fill,inner sep=1pt]
\draw (0,0) node [left] {\scriptsize $\vect{0}{0}$}
-- (1,0) node [right] {\scriptsize $\vect{1}{0}$}
-- (1,1) node [right] {\scriptsize $\vect{1}{1}$}
-- (0,1) node [left] {\scriptsize $\vect{0}{1}$}
-- (0,0)
-- cycle;
\draw [thick,fill=gray!50,fill opacity=0.7]
  (1,1) -- (1,-0.618)
  -- (0,1) -- (1,1) -- cycle;
\draw (0.618,0) node [circle] {}
  node [below left] {\scriptsize $\vect{\phibar}{0}$}
  -- (1,-0.382) node [circle] {}
  node [right] {\scriptsize $\vect{1}{-\phibar^2}$};
\draw (0.382,0.382) node [circle] {}
  node [left] {\scriptsize $\vect{\phibar^2}{\phibar^2}$}
-- (1,0);
\node [circle] at (0,1) {};
\node [circle] at (1,1) {};
\node [circle] at (1,0) {};
\node [circle] at (1,-0.618) {};
\node [right] at (1,-0.618) {\scriptsize $\vect{1}{-\phibar}$};
\node at (0.66,0.66) {\footnotesize $\Om_\infty$};
\node at (0.818,-0.03) {\footnotesize $\Om_\phi$};
\node at (0.93,-0.37) {\footnotesize $\Om_1$};
\end{tikzpicture}
\end{minipage}

\caption{%
Left:
Vectors $w_j$ in $\La$ (the orbit $\Gamma \vv_2$) of slopes $j = 0$, $\phibar$, $1$, 
$\phi$, $\infty$.\newline
Right:
The partition of $\Om$ into $\Om_1$, $\Om_\phi$, $\Om_\infty$.
}
\label{fig:sc:zones}
\end{figure}

To calculate the return time and the return map, we need to
understand the vector in $g_{a, b} \La$ of smallest positive slope in the
vertical strip $U_1$.  In Figure~\ref{fig:sc:zones}, we show a selection of 
vectors in $\La$.  Particularly relevant to our discussion are the
vectors $w_1 = (\phi,\phi)$, $w_\phi = (1,\phi)$, $w_\infty = (0,1)$.

Consider the partition of $\Om$ into the
three subdomains $\Om_1$, $\Om_\phi$, $\Om_\infty$ 
defined by:
\begin{align*}
\Om_1 & = \bigl\{\, (a,b) \in \Om :\quad
\phibar \le a \le 1,\quad 1- a \phi < b \le \phibar-a \,\bigr\}\\
\Om_\phi & = \bigl\{\, (a,b) \in \Om :\quad
\phibar^2 \le a \le 1,\quad \phibar-a < b \le \phibar-a\phibar \,\bigr\}\\
\Om_\infty & = \bigl\{\, (a,b) \in \Om :\quad
0 < a \le 1,\quad \phibar-a\phibar < b \le 1 \,\bigr\}
\end{align*}
These subdomains are illustrated in Figure~\ref{fig:sc:zones}.

A direct calculation, left as an exercise to the reader, shows that
for each $j$ in $\{ 1, \phi, \infty \}$, for any $(a, b)$ in $\Om_j$,
the vector $g_{a, b} w_j$ is the one with smallest slope in
$g_{a, b} \La \cap U_1$.


%
%
%

In each zone, the return time is given by the slope of the corresponding
vectors.  Thus,

\psp
\tbx if $(a,b)\in\Om_1$,\quad
$g_{a, b} w_1 = \phi( (a+b) + a^{-1} i)$,\quad and\quad
$R(a, b) = \frac{a^{-1}}{a+b} \frac{1}{a(a+b)}$;

\psp
\tbx if $(a,b)\in\Om_\phi$,\quad
$g_{a, b} w_\phi = (a+ b\phi) + a^{-1}\phi i $,\quad and \quad
$R(a, b) = \frac{a^{-1}\phi}{a+b\phi} = \frac{1}{a(a\phibar + b)}$;

\psp
\tbx if $(a,b)\in\Om_1$,\quad
$g_{a,b} w_\infty = b + a^{-1}i$,\quad and \quad
$R(a, b) = \frac{a^{-1}}{b} = \frac{1}{ab}$.

To compute the return map, Note that $h_s g_{a,b} = \mat{a}{b}{-s a}{-s b + a^{-1}}$ is 
then respectively 
$\mat{a}{b}{\frac{-1}{a+b}}{\frac{1}{a+b}}$,
$\mat{a}{b}{\frac{-1}{a\phibar+b}}{\frac{\phibar}{a\phibar+b}}$,
$\mat{a}{b}{\frac{-1}{b}}{0}$.

\medskip
Computing the canonical representative $g_{T(a,b)}$
in the class $h_s g_{a,b} \Gamma$ yields:

\pmp
\tbx in $\Om_1$:
$\displaystyle
\mat{a}{b}{\frac{-1}{a+b}}{\frac{1}{a+b}}
\mat{\phi}{1}{\phi}{\phi}
u_\phi^{k_1}
= \mat{(a+b)\phi}{a+b\phi+k_1\phi^2(a+b)}{0}{\frac{\phibar}{a+b}}$,
$k_1 = -\left\lfloor  \frac{a+b\phi-1}{\phi^2(a+b)} 
\right\rfloor$,

\pmp
\tbx in $\Om_\phi$:
$\displaystyle
\mat{a}{b}{\frac{-1}{a\phibar+b}}{\frac{\phibar}{a\phibar+b}}
\mat{1}{0}{\phi}{1}
u_\phi^{k_phi}
= \mat{a+b\phi}{b+k_\phi\phi(a+b\phi)}{0}{\frac{\phibar}{a\phibar+b}}$,
$k_\phi = -\left\lfloor  \frac{b-1}{\phi(a+b\phi)} 
\right\rfloor$,

\pmp
\tbx in $\Om_\infty$:
$\displaystyle\mat{a}{b}{\frac{-1}{b}}{0}
\mat{0}{-1}{1}{0}
u_\phi^{k_\infty}
= \mat{b}{-a+k_\infty\phi b}{0}{\frac{1}{b}}$,
$k_\infty = -\left\lfloor  \frac{-a-1}{\phi b} \right\rfloor$.

\pmp
Thus, we have the following formulas for $T(a, b)$:

\tbx in $\Om_1$:
$T(a, b) = \left((a+b)\phi, a+b\phi+k_1\phi^2(a+b)\right)$,
$k_1 = -\left\lfloor  \frac{a+b\phi-1}{\phi^2(a+b)} 
\right\rfloor$,

\tbx in $\Om_\phi$:
$T(a, b) = \left( a+b\phi, b+k_\phi\phi(a+b\phi)\right)$,
$k_\phi = -\left\lfloor  \frac{b-1}{\phi(a+b\phi)} 
\right\rfloor$, 

\tbx in $\Om_\infty$:
$T(a, b) = (b, -a+k_\infty\phi b)$,
$k_\infty = -\left\lfloor  \frac{-a-1}{\phi b} \right\rfloor$.

\medskip
\noindent This completes the proof of Theorem~\ref{theorem:cross}.\qed


\section{Ergodicity and equidistribution}
\label{subsec:ergodic} 

\subsection{Ergodic theory for $T$}
The construction of the map $T$ as a first return map for horocycle flow on $X
= \SL(2, \R)/\Gamma$ allows us to classify the ergodic invariant measures for
$T$ and that long periodic orbits of $T$ equidistribute, as consequences of
the corresponding results for $h_s$ acting on $\SL(2, \R)/\Gamma$, due to
Dani-Smillie~\cite{DS}.

\begin{Theorem}
\label{theorem:ergodic}
The Lebesgue probability measure $m$ given by $dm = \frac{2}{\phi} da db$ is
the unique ergodic invariant probability measure for $T$ not supported on a
periodic orbit.  In particular it is the unique ergodic absolutely continuous
invariant measure (acim).  For every $(a, b)$ not periodic under $T$ and any
function $f \in L^1(\Om, dm)$, we have that $$\lim_{N \rightarrow \infty}
\frac{1}{N} \sum_{i=0}^{N-1} f(T^i (a, b)) = \int_{\Om} f dm$$ Moreover, if
$\{(a_r, b_r)\}_{r=1}^{\infty}$ is a sequence of periodic points with periods
$N(r) \rightarrow \infty$ as $r \rightarrow \infty$, we have, for any bounded
function $f$ on $\Om$, $$\lim_{r \rightarrow \infty} \frac{1}{N(r)}
\sum_{i=0}^{N(r)-1} f(T^i (a_r, b_r)) = \int_{\Om} f dm.$$
\end{Theorem}

\begin{proof}
By standard theory of first return maps, if $(a,b)$ is not periodic with
respect to $T$, then $g_{a, b} \Gamma$ is not $h_s$-periodic.  By the results
of Dani-Smillie~\cite{DS}, non-periodic orbits of $h_s$ equidistribute (i.e.,
are Birkohff regular) with respect to Haar measure $\mu$ on $\SL(2,
\R)/\Gamma$, which can be described as (a multiple of) $da db ds$ when we
realize $\SL(2, \R)/\Gamma$ as a suspension space over $\Om$ (see also
Appendix~\ref{appendix:volume} for detailed volume computations).  Thus, the
corresponding non-periodic orbit of $T$ on $\Om$ must equidistribute with
respect to Lebesgue measure on $\Om$.  Similarly, by results of Sarnak~\cite{Sarnak}, long periodic horocycles on
$\SL(2, \R)/\Gamma$ equidistribute, and thus, long periodic orbits of $T$ on
$\Om$ must equidistribute. Since the roof function $R$ is bounded below by
$1$, the length of the discrete period $N(r) \rightarrow \infty$ implies that
the length of the continuous period of $g_{a_r, b_r} \Gamma$ must also go to
infinity (in fact, it is at least $N(r)$).
\end{proof}

\subsection{Proof of main theorems}
\label{subsec:mainproof}

Before proving Theorem~\ref{theorem:main},
we state and prove a more general result:

\begin{Theorem}
\label{theorem:general}
Let $x = g\Gamma \in X$ be such that $x$ is not $h_s$-periodic.  Let
$$
\s_g = \{ 0 \le s_1 < s_2 < \ldots < s_N < \ldots\}$$ be the set of slopes of
elements of $g\La$ in the vertical strip $U_1$. Let
$$\G^N_x = \{s_{i+1} - s_i: i =1, 2, \ldots, N\}$$
denote the associated gap set.  Then for any $t \geq 0$,
$$
\lim_{N
\rightarrow \infty} \frac{1}{N}
\Bigl\lvert\G^N_x \cap [t, \infty)\Bigr\rvert
= m(\bigl\{\,(a, b) \in \Om: R(a, b) \geq t \,\bigr\}).
$$
If $x = g\Gamma$ is $h_s$-periodic, define $x_r = a_{-r}
g\Gamma=a_{-r}x$.  Then there is a $P(r)$ so that for any $N \geq P(r)$,
$\G^N_{x_r} = \G^{P(r)}_{x_r}.$ We then have
$$
\lim_{r \rightarrow \infty} \frac{1}{P(r)}
\Bigl\lvert\G^{P(r)}_{x_r} \cap [t, \infty)\Bigr\rvert
= m( \bigl\{\,(a, b) \in \Om: R(a, b) \geq t\,\bigr\}).
$$
\end{Theorem}

\begin{proof}
As observed in \S\ref{sec:connect}, the proportion $ \frac{1}{N} |\G^N_x \cap
[t, \infty)|$ of the first $N$ slope gaps of size at most $t$ in the strip
$U_1$ is a Birkhoff sum of the indicator function of the super-level set
$\{(a, b) \in \Om: R(a, b) \geq t\}$.  The first statement then follows from
the first statement of Theorem~\ref{theorem:ergodic}.  For the second
statement, note that if $x$ is periodic under $h_s$ with period $s_0$,
$a_{-r}x$ is periodic with period $e^r s_0$, by the conjugation relation $a_r
h_s a_{-r} = h_{se^{-r}}$.  $P(r)$ is the corresponding period for the map
$T$.  The second statement of the theorem then follows from the second
statement of Theorem~\ref{theorem:ergodic}.
\end{proof}

\subsection{Proof of Theorem~\ref{theorem:main}} 
\begin{lemma}
Let $x$ be the golden L and $x_R=a_{-2\log R}x$. We have $$\frac{\left| \G_R(\La) \cap [t,
\infty)\right|}{N(R)} = \frac{1}{N(R)} |\G^{N(R)}_{x_R} \cap [t, \infty)|.$$
\end{lemma}
\begin{proof}
The golden L is $h_s$-periodic, with period $\phi$, since the matrix
$\mat{1}{0}{-\phi}{1} \in \Gamma$.  We are interested in the slopes of saddle
connections with horizontal component at most $R$, and slope at most $1$.  The
map $T$ does not see any of the saddle connections slopes for vectors of
length more than $1$.  However, renormalizing by the matrix $a_{-2\log R} =
\mat{1/R}{0}{0}{R}$, we can consider the point $x_R = a_{-2\log R}\Gamma \in
X$.  Note that this matrix scales slopes and thus differences of slopes by
$R^2$.  Each $\gamma \in \s_R$ corresponds with a saddle connection $a_{-2\log
R}\gamma$ on $x_R$ which is in $\Om$.  The point $x_R$ has period $N(R)$
under $T$.
\end{proof}

\begin{proof}[Proof of Theorem~\ref{theorem:main}]
By Theorem \ref{theorem:general} we have
$$
\frac{1}{N(R)} \,
\Bigl\lvert \G^{N(R)}_{x_R} \cap (t,\infty)\Bigr\rvert
\to  m(\bigl\{ (a, b) \in \Om: R(a, b) \geq t\bigr\}).
$$
So by the previous lemma
$$
\frac{1}{N(R)}\x
\bigl\lvert \G_R(\La) \cap [t, \infty)\bigr\rvert
\to m(\bigl\{ (a, b) \in \Om: R(a, b) \geq t\bigr\}).
$$
\end{proof}


\subsection{Spacings and statistics}
\label{subsec:statistics}

The equidistribution of periodic points also yields significant further
information on higher-order spacings and statistics for the gap distribution.
We record one representative result on $h$-spacings, and refer the reader
to~\cite[\S1.5]{ACheung} for further results of this type in the setting of
the torus, whose proofs can be easily modified to this setting.

\begin{Theorem}
\label{theorem:hspacing}
Let $h$ be a positive integer, and let $t_1, t_2, \ldots, t_h >0$.  Then the
$h$-spacing distribution
$$
\frac{1}{N(R)}\x
\Bigl\lvert  \Bigr\{ \, 1 \le i \le N(R): R^2 (s^{(i+k)}_R -
s^{(i+k-1)}_R) > t_k,\quad 1 \le k \le h \, \Bigr\} \Bigr\rvert
$$
converges, as $R \to \infty$, to
$$
m(\bigl\{ (a, b) \in \Om:  R(T^j(a, b)) \geq t_j, 0 \le j < h\bigr\})
$$
\end{Theorem}

\begin{proof}
Apply the periodic case of Theorem~\ref{theorem:general} to the indicator
function of the set
$$
\bigl\{(a, b) \in \Om: R(T^j(a, b)) \geq t_j, 0 \le j < h\bigr\}.
$$
\end{proof}


\section{Further questions}

Our method to explicitly compute the gap distribution in this setting leads us
to several natural questions.

\subsection{Real analyticity}

\begin{ques}
Is the gap distribution for a generic (with respect to Masur-Veech measure on
a stratum $\hh(\kk)$ surface real analytic?
\end{ques}

\noindent
This distribution was shown to exist in~\cite{AChaika}, and
in~\cite{Athreya}, a possible method of computing it by constructing a first
return map for the action of $h_s$ on the stratum $\hh(\kk)$ was suggested.
However, explicitly computing this return map (or indeed the return time,
which is the only requirement for the gap distribution) seems difficult.
Perhaps a property like real analyticity is within reach.  Our computation
here shows that the distribution for the golden L is piecewise real-analytic.
 For the torus, this was computed by~\cite{Hall}, and
was real-analytic on 3 pieces.  This leads to the natural

\begin{ques}
Is the gap distribution for all lattice surfaces piecewise real analytic?
\end{ques}

\noindent
We conjecture the answer is yes, and that the number of pieces are some
measure of the `complexity' of the Veech surface.

\subsection{Support at 0}

In~\cite{AChaika} it was shown that lattice surfaces have no small
(normalized) gaps, and that in contrast, that the gap distribution for generic
surfaces has support at $0$.  This leaves open the question:
\begin{ques} Is
there a translation surface whose gap distribution has no support at $0$ but
does have small gaps.  That is, for all $\epsilon >0$, for all $R \gg 0$,
there exists a gap of slopes of saddle connections of length at most $R$ less
than $\epsilon/R^2$, but there is an $\epsilon_0 >0$ so that the proportion of
gaps of size at most $\epsilon_0/R^2$ goes to $0$ as $R \rightarrow \infty$.
\end{ques}

\noindent
A possible set of candidates for such a surface might be \emph{completely
periodic} surfaces which are not lattice surfaces.

\medskip
\centerline{\tikz \draw (0,0) -- (3,0);}


\appendixmode


\section{Explicit formulas for the probability distribution function}
\label{appendix:formula}

\pmp

\begin{Notation}
\tbx
We use a bar to denote the inverse, so $\alphabar$ 
and $\phibar$ denote $\alpha^{-1}$ and $\phi^{-1}$.
\\
\tbx We denote by $\ath$ the inverse hyperbolic tangent function:
$\ath(x)
= (1/2)\ln((1+x)/(1-x))$.
\\
\tbx We denote by $r$ the function $x \mapsto r(x) = \sqrt{1-4 \x x}$.
\end{Notation}

The cumulative distribution function for the gaps in slopes is the function 
$F$ which to $\alpha$ associates the probability $F(\alpha)$ that a gap
is less than $\alpha$.
The probability distribution function $f$ is the derivative of $F$.

We can see $F$ as the sum of three partial cdfs $F_s$
for the zones $\Om_s$, for each $s\in\{1,\phi,\infty\}$, where
$F_s(\alpha) = \area\bigl(\Om_s\cap(R_s(x,y) < \alpha)\bigr)$,
giving the formulas
$F_\infty(\alpha) = \area\bigl(\Om_\infty\cap (xy > \alphabar)\bigr)$,
$F_\phi(\alpha) = \area\bigl(\Om_\phi\cap(x(\phibar x+y)>\alphabar)\bigr)$, 
and $F_1(\alpha) = \area\bigl(\Om_1\cap (x(x+y)>\alphabar)\bigr)$.

\pmp

The different configurations, as $\alpha$ varies, of the intersection of
the hyperbolas $R_s(x,y)=\alpha$ with the domains $\Om_s$ determine
different evaluations of these formulas; which by differentiating 
give the partial pdfs:

\pmp

\begin{longtable}[l]{@{}r@{}l@{}}
\tbx Respectively for:\quad
&
$0 < \alpha < 1$,\quad
$1<\alpha<4\phi$,\quad
$4\phi<\alpha<\phi^4$,\quad
$\phi^4<\alpha$,
\\
\end{longtable}

\begin{longtable}[l]{@{}r@{}l@{}}
\phantom{\tbx}$F_\infty(\alpha)$ equals:\quad
&
$0$,\quad
$1-\alphabar(1+\ln\alpha)$,\quad
$1-\alphabar(1+\ln\alpha-4\ath r(\phi\alphabar))
-\phibar \x r(\phi\alphabar)$,\quad\\
& $3\phibar/2-\alphabar(1+2\ln(\phibar\alpha/2)+
2\ln(1-r(\phi\alphabar)))+(\phibar/2) r(\phi\alphabar)$;
\\
\end{longtable}

\begin{longtable}[l]{@{}r@{}l@{}}
\phantom{\tbx}$f_\infty(\alpha)$ equals:\quad
&
$0$,\quad
$\alphabar^2\ln\alpha$,\quad
$\alphabar^2(\ln\alpha-4\ath r(\phi\alphabar))$,\quad
$\alphabar^2(2\ln(\phibar\alpha/2)+2\ln(1-r(\phi\alphabar)))$.
\\
\end{longtable}

\psp

\begin{longtable}[l]{@{}r@{}l@{}}
\tbx Respectively for:\quad
&
$0<\alpha<\phi$,\quad
$\phi<\alpha<4$,\quad
$4<\alpha<\phi^3$,\quad
$\phi^3<\alpha$,
\\
\end{longtable}

\begin{longtable}[l]{@{}r@{}l@{}}
\phantom{\tbx}$F_\phi(\alpha)$ equals:\quad
&
$0$,\quad
$\phibar-\alphabar(1+\ln{\phibar\alpha})$,\quad
$\phibar-\alphabar(1+\ln{\phibar\alpha}-
4\ath r(\alphabar))- r(\alphabar)$,\quad
$\phibar^4$;
\\
\end{longtable}

\begin{longtable}[l]{@{}r@{}l@{}}
\phantom{\tbx}$f_\phi(\alpha)$ equals:\quad
&
$0$,\quad
$\alphabar^2\ln{\phibar\alpha}$,\quad
$\alphabar^2(\ln{\phibar\alpha}-4\ath r(\alphabar))$,\quad
$0$.
\\
\end{longtable}

\psp

\begin{longtable}[l]{@{}r@{}l@{}}
\tbx Respectively for:\quad
&
$0<\alpha<\phi$,\quad
$\phi<\alpha<4\phibar$,\quad
$4\phibar<\alpha<\phi^2$,\quad
$\phi^3<\alpha$,
\\
\end{longtable}

\begin{longtable}[l]{@{}r@{}l@{}}
\phantom{\tbx}$F_1(\alpha)$ equals:\quad
&
$0$,\quad
$\phibar-\alphabar(1+\ln{\phibar\alpha})$,\quad
$\phibar-\alphabar(1+\ln{\phibar\alpha}-
2\ath r(\phibar\alphabar))-(\phi/2)\x r(\phibar\alphabar)$,\quad
$\phibar^5/2$;
\end{longtable}

\begin{longtable}[l]{@{}r@{}l@{}}
\phantom{\tbx}$f_1(\alpha)$ equals:\quad
&
$0$,\quad
$\alphabar^2\ln{\phibar\alpha}$,\quad
$\alphabar^2(\ln{\phibar\alpha}-2\ath r(\phibar\alphabar))$,\quad
$0$.
\\
\end{longtable}

%
%
%
%
%
%
%
%


\section{Volume computations}
\label{appendix:volume}

Define a measurable partition of $X = \SL(2,\R) / \Gamma$ into $X_1$, 
$X_\phi$, $X_\infty$ where each $X_j$ is the part of $X$
spanned by $\Om_j$ under the flow $(h_s)$,
until its first return to $\Om$.
The complement of the union of the $X_j$'s is
the union of periodic orbits for the flow $(h_s)$,
which has measure zero.
The partial volumes $V_j = \vol(X_j) = \int_{\Om_j} R$
are obtained by integrating the return time function $R(a, b)$
over the domains $\Om_j$.

\pmp
\noindent\tbx Integrating $R$ over $\Om_1$,
\begin{align*}
V_1
& = \int_{a=\phibar}^{1} \int_{b = 1-a\phi}^{\phibar-a} 
\dfrac{1}{a(a+b)}\; da\,db 
= \int_{\phibar}^{1}
\Bigl[\ln(a+b)\Bigr]_{1-a\phi}^{\phibar-a}\;\dfrac{da}{a} 
= \int_{\phibar}^{1}
\Bigl(\ln(\phibar) - \ln(1-a\phibar)\Bigr)\;\dfrac{da}{a} \\
& = -(\ln \phibar)^2 - \int_{\phibar^2}^{\phibar}
\ln(1-t) \; \dfrac{dt}{t} = -(\ln \phi)^2 + \Li_2(\phibar) - \Li_2(\phibar^2),
\end{align*}
where the last step uses
$\ln \phibar = -\ln \phi$ and the definition
$\Li_2(t) = \int_{t}^{0} \ln(1-t) \; dt/t$.

\psp

\noindent\tbx Integrating $R$ over $\Om_\phi$,

\begin{align*}
V_\phi 
& =
\int_{a=\phibar^2}^{\phibar} \int_{b = 1-a\phi}^{\phibar-a\phibar} 
\dfrac{1}{a(a\phibar+b)}\; da\,db
+\int_{a=\phibar}^{1} \int_{b = \phibar-a}^{\phibar-a\phibar} 
\dfrac{1}{a(a\phibar+b)}\; da\,db \\
& =
\int_{\phibar^2}^{\phibar}
\Bigl[\ln(a\phibar+b)\Bigr]_{1-a\phi}^{\phibar-a\phibar}\;\dfrac{da}{a}
+\int_{a=\phibar}^{1}
\Bigl[\ln(a\phibar+b)\Bigr]_{\phibar-a}^{\phibar-a\phibar}\;\dfrac{da}{a} \\
& =
\int_{\phibar^2}^{\phibar}\Bigl(\ln(\phibar)-\ln(1-a)\Bigr)\;\dfrac{da}{a}
+\int_{\phibar}^{1}\Bigl(\ln(\phibar)-\ln(\phibar-a\phibar^2)\Bigr)\;\dfrac{da}{a} \\
& =
(\ln\phibar)(\ln\phibar-\ln\phibar^2)
-\int_{\phibar^2}^{\phibar}\ln(1-t)\;\dfrac{dt}{t}
-\int_{\phibar^2}^{\phibar}\ln(1-t)\;\dfrac{dt}{t} \\
& =
-(\ln\phibar)^2-2\int_{\phibar^2}^{\phibar}\ln(1-t)\;\dfrac{dt}{t}
= -(\ln\phi)^2-2\Li_2(\phibar^2)+2\Li_2(\phibar).
\end{align*}

\psp

\noindent\tbx Integrating $R$ over $\Om_\infty$,
\begin{align*}
V_\infty
& =
\int_{a=0}^{\phibar^2} \int_{b = 1-a\phi}^{1} 
\dfrac{1}{ab}\; da\,db
+\int_{a=\phibar^2}^{1} \int_{b = \phibar-a\phibar}^{1} 
\dfrac{1}{ab}\; da\,db \\
& =
\int_{0}^{\phibar^2}
\Bigl[\ln b\Bigr]_{1-a\phi}^{1}\;\dfrac{da}{a}
+\int_{\phibar^2}^{1}
\Bigl[\ln b\Bigr]_{\phibar-a\phibar}^{1}\;\dfrac{da}{a} \\
& =
-\int_{0}^{\phibar}\ln(1-t)\;\dfrac{dt}{t}
+(\ln\phibar)(\ln\phibar^2)
-\int_{\phibar^2}^{1}\ln(1-a)\;\dfrac{da}{a} \\
& =
2(\ln\phi)^2+\Li_2(\phibar)
-\int_{\phibar^2}^{0}\ln(1-t)\;\dfrac{dt}{t}
-\int_{0}^{1}\ln(1-t)\;\dfrac{dt}{t} \\
& =
2(\ln\phi)^2+\Li_2(\phibar)-\Li_2(\phibar^2)+\Li_2(1).
\end{align*}

\pmp
\noindent
These three partial volumes add up to
$$
V = \Li_2(1) + 4(\Li_2(\phibar) - \Li_2(\phibar^2)),
$$
which, since\quad $\Li_2(1) = \pi^2/6$\quad and\quad
$\Li_2(\phibar) - \Li_2(\phibar^2) = \pi^2/30$,\quad
can be expressed as
$$
V = \pi^2/6+4\x\pi^2/30 = 9\pi^2/30 = 3\pi^2/10.
$$
This is exactly the classically known volume of
$X = \SL(2, \R)/\Gamma$, as should be expected.




\end{document}